\documentclass[11pt]{amsart}
\usepackage{etex}
\usepackage{amsfonts,amsmath,amssymb,amsthm,amscd,amsxtra}
\usepackage{enumerate,verbatim,color}

\usepackage{amsthm}
\usepackage[leqno]{amsmath}
\usepackage[all]{xy} \SelectTips{eu}{}
\usepackage{latexsym,amsfonts,amssymb}
\usepackage{mathptmx}
\usepackage{array}
\usepackage{xcolor}

\usepackage[pagebackref]{hyperref}

\theoremstyle{plain} 

\newtheorem{thm}{Theorem}[section]

\newtheorem{prop}[thm]{Proposition}
\newtheorem{cor}[thm]{Corollary}

\theoremstyle{definition}
\newtheorem{lem}[thm]{Lemma}

\newtheorem{eg}[thm]{Example}

\newtheorem{ques}[thm]{Question}

\newtheorem{rmk}[thm]{Remark}

\numberwithin{equation}{section}
\newcommand{\lra}{\longrightarrow}

\newcommand{\xla}{\xleftarrow}
\newcommand{\xra}{\xrightarrow}

\newcommand{\fm}{\mathfrak{m}}
\newcommand{\fn}{\mathfrak{n}}
\newcommand{\fp}{\mathfrak{p}}
\newcommand{\fq}{\mathfrak{q}}
\newcommand{\up}[1]{{{}^{#1}\!}}

\DeclareMathOperator{\e}{e}

\DeclareMathOperator{\s}{s}
\DeclareMathOperator{\md}{\mathsf{mod}}

\DeclareMathOperator{\crs}{crs}

\DeclareMathOperator{\Tr}{\mathsf{Tr}}

\newcommand{\drs}{\operatorname{drs}_p}

\DeclareMathOperator{\depth}{\mathsf{depth}}

\DeclareMathOperator{\Ext}{\mathsf{Ext}}

\DeclareMathOperator{\Hom}{Hom}

\DeclareMathOperator{\id}{id}

\DeclareMathOperator{\len}{length}

\DeclareMathOperator{\pd}{pd}

\DeclareMathOperator{\Supp}{Supp}
\DeclareMathOperator{\Spec}{Spec}

\DeclareMathOperator{\Tor}{\mathsf{Tor}}

\newcommand{\vf}{\varphi}

\newcommand{\ov}{\overline}

\newcommand{\wh}{\widehat}
\newcommand{\x}{\underline{\mathbf{x}}}
\newcommand{\F}{\mathbf{F}}

\def\urltilda{\kern -.15em\lower .7ex\hbox{\~{}}\kern .04em}
\def\urldot{\kern -.10em.\kern -.10em}\def\urlhttp{http\kern -.10em\lower -.1ex
\hbox{:}\kern -.12em\lower 0ex\hbox{/}\kern -.18em\lower 0ex\hbox{/}}

\newtheorem{chunk}[thm]{\hspace*{-1.065ex}\bf}

\begin{document}

\title[Tensoring with the Frobenius Endomorphism]{Tensoring with the Frobenius Endomorphism}

\subjclass[2000]{13D07, 13H10}


\thanks{Sadeghi's research was supported by a grant from IPM}

\author[Celikbas, Sadeghi, Yao]
{Olgur Celikbas, Arash Sadeghi and Yongwei Yao}

\address{Olgur Celikbas\\
Department of Mathematics \\
West Virginia University\\
Morgantown, WV 26506-6310, U.S.A}
\email{olgur.celikbas@math.wvu.edu}

\address{Arash Sadeghi\\
School of Mathematics, Institute for Research in Fundamental Sciences (IPM), P.O. Box: 19395-5746, Tehran, Iran}
\email{sadeghiarash61@gmail.com}

\address{Yongwei Yao\\
Department of Mathematics and Statistics, Georgia State University, Atlanta, Georgia 30303, U.S.A}
\email{yyao@gsu.edu}


\keywords{Frobenius endomorphism, tensor products of modules, rank and torsion}

\date{\today}

\begin{abstract}
Let $R$ be a commutative Noetherian Cohen-Macaulay local ring that has
positive dimension and prime characteristic. Li proved that the
tensor product of a finitely generated non-free $R$-module $M$ with the
Frobenius endomorphism $\up{\varphi^n}R$ is not maximal Cohen-Macaulay
provided that $M$ has rank and $n\gg 0$. We replace the rank hypothesis with the
weaker assumption that $M$ is locally free on the minimal prime ideals
of $R$. As a consequence, we obtain, if $R$ is a one-dimensional non-regular complete reduced local ring that has a perfect residue field and  prime characteristic, then
$\up{\varphi^n}R \otimes_{R}\up{\varphi^n}R$ has torsion for all $n\gg
0$. This property of the Frobenius endomorphism came as a surprise to us since,
over such rings $R$, there exist non-free
modules $M$ such that $M\otimes_{R}M$ is torsion-free.  
\end{abstract}

\maketitle{}

\setcounter{tocdepth}{1}

\section{Introduction}

Throughout the paper, $R$ denotes a commutative Noetherian ring and $\md R$ denotes the category of all finitely generated $R$-modules.

When $R$ is of prime characteristic $p$ and $M$ is an $R$-module,
${}^{\vf^n}\!M$ denotes the $R$-module obtained from $M$ by
restriction of scalars along $\vf^n$, where ${\vf}\colon R\lra R$ is
the Frobenius endomorphism given by $r\mapsto r^p$. Therefore the
action of $R$ on ${}^{\vf^n}\!M$ is given by $r\cdot m = r^{p^{n}}m$
for $r\in R$ and $m\in M$. On the other hand, we agree that the
$R$-module structure of the tensor product $M\otimes_R{}^{\vf^n}\!R$ 
comes from the right (ordinary) action of $R$ on ${}^{\vf^n}\!R$, i.e., $r \cdot (m \otimes s)=m \otimes (sr)$ for all $r,s \in R$ and $m\in M$. It follows that 
$M\otimes_R{}^{\vf^n}\!R \in \md R$ if $M\in \md R$. 

When $(R, \fm)$ is local of prime characteristic $p$, following \cite{CM}, we set:
\[
\drs(R)=\inf\{p^n \mid \fm^{[p^n]}\subseteq (\x) \text{ for some
  system of parameters } \x \text{ of } R\},
\]
where $\fm^{[p^n]}$ of $\fm$ denotes the ideal generated by the $p^n$th powers of any set of generators of $\fm$.

In this paper we are concerned with the following result of Li \cite{Li}; recall that a module $M\in \md R$ over a Cohen-Macaulay local ring $R$ has (constant) \emph{rank} if there is a nonnegative integer $r$ such that $M_{\fp}\cong R^{\oplus r}$ for all minimal prime ideals $\fp$ of $R$.

\begin{thm}[{Li \cite[2.4]{Li}}] \label{thmint} Let $(R,\fm)$ be a Cohen-Macaulay local ring of positive dimension and prime characteristic $p$ and let $M\in \md R$. Assume $n$ is an integer such that $p^n\geq\drs(R)$ and $M\otimes_R\up{\varphi^n}R$ is maximal Cohen-Macaulay. If $M$ has rank, then $M$ is free.
\end{thm}

Li \cite[2.5]{Li} points out that the following example from \cite{CM}
shows that Theorem \ref{thmint} could fail without the constant rank hypothesis.
 
\begin{eg}[{\cite[2.1.7]{CM}}] \label{eg} Let $R=k[\![x,y]\!]/(x^2)$, where $k$ is a field of characteristic $p$, and let $M=R/(xy)$. Then $\drs(R)=p$ and $M\otimes_R\up{\varphi}R \cong R/(x^py^p)\cong R$. Therefore $M\otimes_R\up{\varphi}R$ is maximal Cohen-Macaulay but $M$ is not free.
\end{eg}

Our motivation comes from the fact that, in Example \ref{eg}, for
minimal prime
$\fp=(x)\in \Spec(R)$, $M_{\fp}$ is not free over $R_{\fp}$. Hence $M$
is not locally free on the minimal prime ideals of $R$, which implies that $M$ does not have rank. Using an entirely different argument from \cite{Li}, we are able to replace the rank hypothesis of Theorem~\ref{thmint} with the weaker condition that $M$ is locally free on the minimal primes of $R$, i.e., $M_{\fp}$ is free over $R_{\fp}$ for all minimal prime ideals $\fp$ of $R$. More precisely we prove:



\begin{thm}\label{t11} Let $(R,\fm)$ be a Cohen-Macaulay local ring of
  positive dimension and prime characteristic $p$ and let $M\in \md
  R$. Assume that $n$ is an integer such that $p^n\geq\e(R)$ and
  $M\otimes_R\up{\varphi^n}R$ is maximal Cohen-Macaulay. If $M_{\fp}$
  is free over $R_{\fp}$ for all minimal prime ideals $\fp$ of $R$,
  then $M$ is free. 
\end{thm}

Let us remark that the condition $M$ is locally free on the minimal
prime ideals of $R$ holds, for example, when $R$ is reduced, even if
$M$ does not have rank. Moreover, Example~\ref{eg} shows that 
the hypothesis of $M$ being locally free on the minimal primes in
Theorem~\ref{t11} cannot be removed. 

In the next section we give a proof of Theorem \ref{t11}. As an
application of our argument, we obtain the following result; see
Corollary \ref{corend} for a more general statement.

\begin{cor} \label{corintro} Let $R$ be a complete reduced non-regular
  Cohen-Macaulay 
  local ring of prime characteristic $p$ that has a perfect residue
  field. Then $\up{\varphi^n}M\otimes_R\up{\varphi^n}R$ is not maximal
  Cohen-Macaulay for any $M \in \md R$ and for all $n\gg
  0$. In particular $\up{\varphi^n}R\otimes_R\up{\varphi^n}R$ is not
  maximal Cohen-Macaulay for all $n\gg 0$. 
\end{cor}

Prior to proceeding for our main argument, we discuss some results from the literature about tensor products of modules and compare them with Theorem \ref{t11} and Corollary \ref{corintro}.

Torsion properties of tensor products of modules over local rings were
initially studied by Auslander \cite{Au}, and Huneke and Wiegand
\cite{HW1}. Although tensor products tend to have torsion, it is not
unnatural for a tensor product $M\otimes_RN$ to be torsion-free, or
maximal Cohen-Macaulay, for some non-free modules $M$ and $N$. In fact,
when $M\otimes_RN$ is maximal Cohen-Macaulay, it even does not force
$M$ and $N$ to be torsion-free, or maximal Cohen-Macaulay, in
general. For example, Huneke and Wiegand \cite[4.7]{HW1} proved that,
over one-dimensional non-Gorenstein domains, there always exist
non-free modules $M$ such that $M\otimes_{R}M$ is maximal
Cohen-Macaulay. On the other hand, over the domain $R=k[\![t^3, t^4,
t^5]\!]$, there exists a module $M\in \md R$ which has {torsion}
such that $M\otimes_R \omega$ is maximal Cohen-Macaulay, where
$\omega$ is the canonical module of $R$; see \cite[4.8]{HW1}. In the
same direction, Constapel \cite[2.1]{Constapel} constructed modules
$M$ and $N$ over the ring $R=k[\![t^8, t^9, t^{10}, t^{11}, t^{12},
t^{13}]\!]$, both of which have {torsion}, such that
$M\otimes_RN$ is maximal Cohen-Macaulay. Let us also remark that
whether or not there are such examples over complete intersection
rings, mainly over those of codimension at least two, is an open
question; see, for example, \cite[2.10]{CW}.  

In general, torsion properties of a tensor product $M\otimes_RN$ are significantly different when $M$ has rank, and when $M$ is locally free on the minimal primes. For example Huneke and Wiegand \cite[3.1]{HW1} proved that, if $M\otimes_RN$ is maximal Cohen-Macaulay over a hypersurface ring $R$, then $M$ or $N$ is free (and hence both $M$ and $N$ are maximal Cohen-Macaulay) if and only if $M$ or $N$ has rank. This result easily fails when modules do not have rank: if $R=k[\![x,y]\!]/(xy)$ (where $k$ is any field, e.g., $k=\mathbb{F}_{p}$), $M=R/(x)$ and $N=R/(x^2)$, then $M\otimes_{R}M \cong M\otimes_{R}N \cong M$ is maximal Cohen-Macaulay, $M$ and $N$ do not have rank but they are locally free on the minimal primes of $R$ (since $R$ is reduced). Theorem \ref{t11} and Corollary~\ref{corintro} show that there are no such examples for the case where $M=\up{\varphi^n}R$ for all $n\gg 0$.

\section{Main result}

In this section we give a proof of our main result, Theorem \ref{t11}. In preparation we prove  two lemmas which seem to be of independent interest. First we recall:

\begin{chunk}[\cite{AB}] \label{AuBrsequence} 
Let $R$ be a local ring and let $M\in \md R$ be a module. For a positive integer $i$, we denote by $\Omega^iM$ the $i$-th syzygy of M, namely, the image of the $i$-th differential map in a minimal free resolution of $M$. As a convention, we set $\Omega^0M=M$.

The \emph{transpose} $\Tr M$ of $M$ is defined as the cokernel of the $R$-dual map $\partial_1^{\ast}=\Hom_{R}(\partial_1,R)$ of the first differential map $\partial_1$ in a minimal free resolution of $M$. Hence there is an exact sequence of the form
$0\rightarrow M^*\rightarrow P_0^*\rightarrow P_1^*\rightarrow \Tr M\rightarrow 0$. Note that the modules $\Omega^iM$ and $\Tr M$ are uniquely determined up to isomorphism, since so is a minimal free resolution of $M$, and there is a stable isomorphism $\Omega^2\Tr M \cong M^{\ast}$. 
\end{chunk}

\begin{lem}\label{p2}
Let $(R,\fm)$ be a local ring and let $M, N \in \md R$.
Assume that $n$ is a positive integer and that the following conditions hold:
\begin{enumerate}[\rm(i)]
\item $M_{\fp}$ is free for all $ \fp \in \Spec R-\{\fm\}$.
\item $\depth(M\otimes_RN)\geq n$.
\item $\depth(N)\geq n-1$.
\end{enumerate}
Then $\Ext^i_R(\Tr M,N)=0$ for for all $i=1, \ldots, n$.
\end{lem}
\begin{proof}
Note that, if $r\geq 1$ and $\Ext^r_R(\Tr M,N)\neq 0$, then it follows from (i) that $\Ext^r_R(\Tr M,N)$ has finite length and hence  $\depth(\Ext^r_R(\Tr M,N))=0$.

We proceed by induction on $n$. Consider the following exact sequence from \cite[2.8]{AB}:
\begin{equation}\tag{\ref{p2}.1}
0\longrightarrow\Ext^1_R(\Tr M,N)\longrightarrow M\otimes_RN\longrightarrow\Hom_R(M^*,N)\longrightarrow\Ext^2_R(\Tr M,N)\longrightarrow0.
\end{equation}

Assume $n=1$. Suppose $\Ext^1_R(\Tr M,N)\neq 0$. It follows from (\ref{p2}.1) that $\depth(M\otimes_RN)=0$, which contradicts (ii). Therefore $\Ext^1_R(\Tr M,N)=0$.

Now assume $n=2$ and consider the following short exact sequence induced from (\ref{p2}.1):
\begin{equation}\tag{\ref{p2}.2}
0\longrightarrow M\otimes_RN\longrightarrow\Hom_R(M^*,N)\longrightarrow\Ext^2_R(\Tr M,N)\longrightarrow 0.
\end{equation}
Suppose $\Ext^2_R(\Tr M,N)\neq 0$. Since $\depth(N)\geq 1$, we have $\depth(\Hom_R(M^*,N))\geq 1$; see, for example, \cite[1.2.28]{BH}. Hence the depth lemma and the exact sequence (\ref{p2}.2) imply that $\depth(M\otimes_RN)=1$, which contradicts (ii). Consequently $\Ext^2_R(\Tr M,N)=0$.

Now assume $n\geq 3$. Then the induction hypothesis yields $\Ext^i_R(\Tr M,N)=0$ for all $i=1, \ldots, n-1$. In particular it follows from  (\ref{p2}.1) that $M\otimes_RN\cong\Hom_R(M^*,N)$. So, by (ii), $\depth(\Hom_R(M^*,N))\geq n$. Therefore it follows from \cite[2.3.3]{D} that $\Ext^i_R(M^*,N)=0$ for all $i=1, \ldots, n-2$. Since $M^{\ast} \cong \Omega^2 \Tr M$, we conclude $\Ext^i_R(\Tr M,N)=0$ for all $i=1, \ldots, n$.
\end{proof}

Our next result uses some techniques of Koh-Lee \cite{KL}. 

\begin{lem}\label{p1}
Let $(R, \fm, k)$ be a Cohen-Macaulay local ring of positive dimension
$d$ and prime characteristic $p$, and let $M \in \md R$. 
Assume that
$n$ is an integer such that $p^n\geq\drs(R)$ and $p^n \geq
\drs(R_\fp)$ for all minimal primes $\fp$ of $R$. If $t$ is a positive
integer and $\Ext^i_R(M,\up{\varphi^n}R)=0$ for all $i=t,\, \dots,\,
t+d-1$, then $\pd(M)<t$. In particular, if $t=1$, i.e., if
$\Ext^i_R(M,\up{\varphi^n}R)=0$ for all $i=1,\, \ldots,\, d$, then $M$
is free. 
\end{lem}

\begin{proof}
By \cite[Theorem~A]{NTY}, for every minimal prime ideal $\fp$ of $R$,
we have $\pd_{R_\fp}(M_\fp) < \infty$, which then implies that $M_\fp$ is free
over $R_\fp$.  

To prove by contradiction, we assume $\pd(M) \ge t$.
Consider a minimal free resolution of $M$ over $R$:
\[
\F=
\dotsb \to F_{t+d} \xra{A_{t+d}} F_{t+d-1} \xra{A_{t+d-1}} \dotsb
\xra{A_{t+1}} F_{t} \xra{A_{t}} F_{t-1} \xra{} \dotsb
\]
Then $F_t \neq 0$ and $A_t \neq 0$. 
Since $R$ is Cohen-Macaulay, all associated primes of $R$ are
minimal. Thus there exists a minimal prime $\fp$ of $R$ such that
$(A_t)_\fp \neq 0$, in which $(A_t)_\fp$ denotes the matrix over
$R_\fp$ naturally derived from $A_t$ via localization at $\fp$. 
This localization process also gives rise to the following free
resolution of $M_\fp$ over $R_\fp$: 
\[
\F_\fp=
\dotsb \to (F_{t+d})_\fp \xra{(A_{t+d})_\fp} (F_{t+d-1})_\fp \xra{(A_{t+d-1})_\fp} \dotsb
\xra{(A_{t+1})_\fp} (F_{t})_\fp \xra{(A_{t})_\fp} (F_{t-1})_\fp \xra{} \dotsb
\]
As $M_\fp$ is free over $R_\fp$, the above resolution $\F_\fp$ is
split exact. Thus the image of $(A_t)_\fp$ is a non-zero $R_\fp$-free
direct summand of $(F_{t-1})_\fp$. This further implies that
$\varphi^n((A_t)_\fp)$ is non-zero. (Indeed, applying the Frobenius functor
to $\F_\fp$ affects neither its split exactness nor the ranks of the
images of the differential maps.)
In particular, $\varphi^n(A_t) \neq 0$. 
Hence $\varphi^n(A_{t}^T) \neq 0$, in which $A_t^T$ denotes the
transpose of the matrix $A_t$. 

Applying $\Hom_R(-,\up{\varphi^n}R)$ to $\F$ and using our assumption,
we get the exact sequence: 
\[
0 \xla{} C \xla{} F_{t+d} \xla{\varphi^n(A_{t+d}^T)} F_{t+d-1} \xla{\varphi^n(A_{t+d-1}^T)} \dotsb
\xla{\varphi^n(A_{t+1}^T)} F_{t} \xla{\varphi^n(A_{t}^T)} F_{t-1} \xla{B} G
\]
in which $C$ is the cokernel of the map $\varphi^n(A_{t+d}^T)$
and $G$ is just a free $R$-module 
such that the above is a free resolution of $C$ with the property that
the rank of $G$ equals the minimal number of generators of the kernel
of $\varphi^n(A_{t}^T)$. 
Notice that all the entries of $\varphi^n(A_{i}^T)$ are in $\fm^{[p^n]}$.
Moreover, by a property of (minimal) free resolutions over a local
ring, we see that, after a proper base change, $B$ can be represented as 
$\left(\begin{smallmatrix}I & 0 \\0 & B'\end{smallmatrix}\right)$ in
block form, in which $I$ is an identity matrix and all the entries of
$B'$ are in $\fm$. 

Next we claim that the row number of $B'$ is positive. Indeed, if 
$B = \left(\begin{smallmatrix}I & 0\end{smallmatrix}\right)$ in block
form, then the image of $B$ is $F_{t-1}$ and hence $\varphi^n(A_{t}^T)
= 0$, which is not the case.

As $p^n\geq\drs(R)$, there exists a system of parameters 
$\x := x_1, \dotsc, x_d$
(hence a maximal $R$-regular sequence) such that $\fm^{[p^n]}
\subseteq (\x)$. Let $E = E_{R/(\x)}(k) \cong \Hom_R(R/(\x),
E_R(k))$. It follows that $\id_R(E) = d = \dim(R)$ and $\fm^{[p^n]}E = 0$.

Applying $\Hom_R(-, E)$ to the above resolution of $C$, we get
\[
0 \xra{} E_{t+d} \xra{\varphi^n(A_{t+d})} 
\dotsb
\xra{\varphi^n(A_{t+1})} E_{t} \xra{\varphi^n(A_{t})} E_{t-1} \xra{B^T} \Hom_R(G, E)
\]
in which both $E_i = \Hom_R(F_i, E)$ and $\Hom_R(G, E)$ are direct
sums of $E$. Since $\Ext^{d+1}_R(C, E)$ vanishes, the above sequence must be
exact at $E_{t-1}$. This is a contradiction, as the 
map $B^T$ is not injective (because of socle elements of $E_{t-1}$ and
the presence of
$\left(\begin{smallmatrix}0\\(B')^T\end{smallmatrix}\right)$ as a
non-trivial part of columns in $B^T$)
while the map $\varphi^n(A_{t})$ is $0$ (because $\fm^{[p^n]}E_t = 0$.) 
\end{proof}

Before giving our proof of Theorem \ref{t11}, we need the following
observation. Note that part of \ref{rmk1} has already been observed in
the proof of \cite[2.2]{DIM}. 

\begin{chunk}\label{rmk1}
Let $(R,\fm,k )$ be a Cohen-Macaulay local ring with an infinite residue field $k$. Then, for all $\fp \in \Spec R$, it follows that
\begin{equation}\tag{\ref{rmk1}.1}
\e(R) \ge \e(R_\fp) = \len_{R_\fp}(R_\fp/(\x(\fp)) \ge \drs(R_\fp),
\end{equation}
where $\e(R)$ is the multiplicity of $R$, and $\x(\fp)$ denotes a minimal reduction of $\fp R_{\fp}$ in $(R_\fp, \fp R_{\fp})$. Note that such a reduction exists since $k(\fp)$ is infinite for each $\fp \in \Spec(R)$.

In (\ref{rmk1}.1) the first inequality and the equality are well-known; see \cite{Lech} and \cite[4.6.8]{BH}, respectively (see also the proof of \cite[2.2]{DIM}.) The second inequality is due to the fact that $\fn^{\len(S/J)} \subseteq J$ for any
$\fn$-primary ideal $J$ of a local ring $(S, \fn)$.
\end{chunk}

We can now give our proof of Theorem \ref{t11}. Recall that a local ring $R$ of prime characteristic $p$ is called \emph{F-finite} if ${}^{\vf}\!R$, viewed as a module via the Frobenius endomorphism $\vf$, is finitely generated; see, for example, \cite[page 398]{BH}.

\begin{proof}[A proof of Theorem \ref{t11}]
We can, if necessary, 
find a local ring extension $(S,\fn)$
of $(R,\fm)$ such that $S$ is F-finite, faithfully flat over $R$, with
an infinite residue field, and $\fm S = \fn$. 
(For example, letting $\wh R \cong k[[x_1,\dotsc,x_m]]/I$ and using
$\ov k$ to denote the algebraic closure of $k$, we can pick
$S = \ov k[[x_1,\dotsc,x_m]]/I\ov k[[x_1,\dotsc,x_m]]$.)
Then $S$ is a Cohen-Macaulay local ring, $\e(R) = \e(S)$, and
$\dim(S) = \dim(R)$.

By going-down, all minimal prime ideals of $S$ contract to minimal primes of
$R$. Thus $M \otimes_R S_\fp$ is free for all minimal prime ideals
$\fp$ of $S$. Moreover
\[
(M \otimes_R S) \otimes_S  \up{\varphi^n}S \cong (M \otimes_R \up{\varphi^n} R)
\otimes_R S
\]
is maximal Cohen-Macaulay over $S$; see \cite[23.3]{Mat}. Note also
that, $M \otimes_R S$ is free over $S$ if and only if $M$ is free over
$R$. Therefore it suffices to prove the case where $R$ is F-finite and with
an infinite residue field. 
By \ref{rmk1}, we have $p^n \ge \drs(R_\fp)$ for all $\fp \in \Spec(R)$.

We now proceed by induction on $d=\dim(R)$. Let $d=1$ and let $M\otimes_R\up{\varphi^n}R$ be maximal Cohen-Macaulay. By Lemma \ref{p2}, $\Ext^1_R(\Tr M, \up{\varphi^n}R)=0$ and so $\Tr M$ is free by Lemma \ref{p1}. Hence $M$ is free. Now assume that $d>1$. By induction hypothesis, $M_{\fp}$ is free for all non-maximal prime ideals $\fp$. Since $M\otimes_R\up{\varphi^n}R$ is a maximal Cohen-Macaulay $R$-module, by Lemma \ref{p2}, we have:
\begin{equation}\notag{}
\Ext^i_R(\Tr M,\up{\varphi^n}R)=0 \text{ for all } i=1, \ldots, d.
\end{equation}
Now the assertion follows from Lemma \ref{p1}.
\end{proof}




We finish this section with an application of Theorem \ref{t11} and
give a proof of Corollary \ref{corintro}. We start by recording a few
preliminary results, the first one being a special case of a result of
Avramov, Hochster, Iyengar and Yao \cite{AHIY}: it is a strengthening
of a classical result of Kunz \cite{Kunz} which considers the case
where $M=R$. 

\begin{chunk}[{\cite[1.1]{AHIY}}] \label{AHIY} Let $R$ be a local ring of prime characteristic $p$ and let $0\neq M\in \md R$. If $R$ is not regular, then $\pd_R(\up{\varphi^i}M)=\infty$ for all $i\geq 1$. 
\end{chunk}

Recall that a module $M\in \md R$ satisfies \emph{Serre's condition} $(S_n)$ if $\depth_{R_\fp}(M_\fp)\geq\min\{n,\dim(R_\fp)\}$ for all $\fp\in\Supp(M)$. 



Note that every complete local ring of prime characteristic $p$ with a
perfect residue field is F-finite; see \cite[page 398]{BH}. Thus we
reach a result that, in particular, establishes Corollary
\ref{corintro} advertised in the introduction. 

\begin{cor} \label{corend} Let $(R, \fm)$ be a reduced, F-finite
  Cohen-Macaulay local ring of positive dimension and prime
  characteristic $p$, and let $0 \neq M\in \md R$. Assume that $s$, $n$ and $t$ are positive integers such that $n\geq\e(R)$. If $\up{\varphi^s}M\otimes_R\up{\varphi^n}R$ satisfies Serre's condition $(S_t)$, then $R_\fq$ is regular for all $\fq \in \Supp(M)$ with $\dim(R_{\fq})\leq t$. In particular, if $\up{\varphi^s}M\otimes_R\up{\varphi^n}R$ is maximal Cohen-Macaulay, then $R$ is regular.
\end{cor}

\begin{proof} 


Suppose that $\up{\varphi^s}M\otimes_R\up{\varphi^n}R$ satisfies
Serre's condition $(S_t)$ and let $\fq \in \Supp(M)$ with
$\dim(R_{\fq})\leq t$. As $R_\fq$ is clearly regular when $\dim(R_\fq)
= 0$, we further assume $\dim(R_\fq) > 0$. By (\ref{rmk1}), we have that $p^n\geq\e(R_\fq)$. Moreover, since $\up{\varphi^s}M\otimes_R\up{\varphi^n}R$ satisfies Serre's condition $(S_t)$, it follows that $\up{\varphi^s}M_{\fq}\otimes_R{_\fq}\up{\varphi^n}R_{\fq}$ is a maximal Cohen-Macaulay $R_{\fq}$-module. Hence we conclude from Theorem \ref{t11} that $\up{\varphi^s}M_{\fq}$ is free over $R_{\fq}$. So the conclusion follows from \ref{AHIY}.
\end{proof}

\section*{Acknowledgments}
We are grateful to Tokuji Araya for his feedback on the manuscript. 

\bibliographystyle{amsplain}

\end{document}